\documentclass[10pt,preprint]{amsart}
\usepackage{amssymb,amsfonts,amsthm,amsmath,amssymb,amscd,relsize}
\usepackage{hyperref}
\usepackage{enumerate}
\usepackage{comment}
\usepackage{xcolor}
\usepackage[normalem]{ulem}

 \usepackage{epsfig}

\usepackage[a4paper,
            bindingoffset=0.2in,
            left=1in,
            right=1in,
            top=1in,
            bottom=1in,
            footskip=.25in]{geometry}

\usepackage{blindtext}

\def\endemo{\hfill\hfill\qed\end{proof}}
\def\tr{\operatorname{tr}}

\def\ad{\operatorname{ad}}
\def\Span{\operatorname{Sp}}

\def\f{\mathfrak f}
\def\g{\mathfrak g}
\def\aa{\mathfrak a}
\def\m{\mathfrak m}
\def\q{\mathfrak q}

\def\V{\mathcal V}
\def\bV{\overline V}
\def\bH{\overline H}
\def\h{\mathfrak h}

\def\R{\mathbb R}
\def\vp{V^{\perp}}

\def\h{\mathfrak h}

\def\m{\mathfrak m}

\theoremstyle{plain}
\newtheorem{theorem}{Theorem}
\newtheorem{proposition}{Proposition}
\newtheorem{lemma}{Lemma}
\newtheorem{corollary}{Corollary}
\newtheorem{question}{Question}
\newtheorem{problem}{Problem}

\theoremstyle{definition}
\newtheorem{definition}{Definition}

\theoremstyle{remark}
\newtheorem{remark}{Remark}
\newtheorem{example}{Example}

\begin{document}

\title{Respectful decompositions  of Lie algebras}
\author{Grant Cairns and Yuri Nikolayevsky}
\address{Department of Mathematical and Physical Sciences, La Trobe University, Melbourne, Australia 3086}
\email{G.Cairns@latrobe.edu.au, Y.Nikolayevsky@latrobe.edu.au}

%\keywords foliation,  geodesic, integrability, sub-Riemannian
%\subjclass Primary 57R30 53C12

\begin{abstract}
One of Pierre Molino's principal mathematical achievements was his theory of Riemannian foliations. One of his last papers, published in 2001, showed that his theory could be extended to a large class of non-integrable distributions. The key example here is that of a \emph{respectful decomposition} of a Lie algebra $\g$; this is vector space decomposition $\g=H+V$ such that $[V,H]\subseteq H$. This paper will examine the basic properties of respectful decompositions.
\end{abstract}

\maketitle

\vskip-0.2in
%\hsize=6truein
\centerline{
\emph{Dedicated with affection  to the memory of Pierre Molino\footnote{This is an expanded version of a talk given by the first author at the \emph{Journ\'ees de G\'eom\'etrie Diff\'erentielle
en m\'emoire de Pierre Molino}, held on 16--17 June 2022}}
}
\vskip-0.5in

%\hsize=6truein

\section{Introduction}

It was shown in \cite{CM} that Molino's theory of Riemannian foliations extends  to a large class of non-integrable distributions. In the general case, a very similar procedure is pursued to the Riemannian case. We will not go into this construction here; details can be found in \cite{CM}. However the basic idea is to first lift to the transverse frame bundle, and then restrict to the orbit closures. In the Riemannian foliation case this ultimately reduces to the study of Lie foliations, in the sense of Fedida \cite{F}. In the general case, one is lead to the similar situation, the simplest form of which is provided by the following simple homogeneous structure.

\begin{definition} An ordered pair $(H,V)$ of complementary vector subspaces   of a Lie algebra $\g$ are said to form a \emph{respectful decomposition} of $\g$ if $[V,H]\subseteq H$.
In this case, we say that $V$ \emph{respects} $H$.
\end{definition}

On a Lie group $G$ corresponding to $\g$, we identify $H$ and $V$ with left-invariant distributions on $G$. Suppose  $V$ respects $H$ and consider a left-invariant Riemannian metric for which $H$ is orthogonal to $V$. 
When $H$ is a subalgebra, $H$ defines a Riemannian foliation on  $G$; in fact, $H$ defines a transversally parallelizable foliation \cite{M}.
When $V$ is a subalgebra, $V$ defines a totally geodesic foliation on $G$; in fact, $V$ defines a tangentially parallelizable foliation \cite{C}.
In the general case, even when $V$ is not a subalgebra,  $V$ is totally geodesic in the sense that every geodesic of $G$ tangent to $V$ at one point will remain everywhere tangent to $V$  \cite{CM}. For a nice account of totally geodesic non-integrable distributions, see \cite{ML}.

Note that when $H$ is a subalgebra, the \emph{respectful decomposition} hypothesis implies that $H$ is an ideal. Conversely, if $H$ is an ideal, then 
$V$  respects $H$ for every choice of complementary subspace $V$. So for respectful decompositions, the case where $H$ is a subalgebra is well understood.

In the case where $V$ is a subalgebra, the situation is nontrivial. There are several known results, especially in the nilpotent case, and there are several open problems. We will survey the situation in Section~\ref{S:tg}.
The main focus of this paper is to make a preliminary examination of the case where $\g$ is nilpotent and neither $H$ nor $V$ is a subalgebra. 

\begin{definition}
Suppose that $(H,V)$ is a respectful decomposition of a Lie algebra $\g$. We say that $(H,V)$ is \emph{open} if neither $H$ nor $V$ is a Lie subalgebra of $\g$.
\end{definition}

If $(H,V)$ is open, then obviously $\dim V\ge 2$.
The following theorem shows that the condition for a nilpotent Lie algebra $\g$ to admit an open respectful decomposition $(H, V)$ with $\dim V = 2$ can be expressed solely in terms of $V$.

\begin{theorem}\label{T:v2} Suppose that a nilpotent Lie algebra $\g$ has an open respectful decomposition $(H,V)$ with $\dim V = 2$. Then the following conditions hold:
\begin{enumerate}
\item \label{it:VV}
$[V,V] \ne 0$, 
\item \label{it:VVg} 
$V \cap [V,\g] =0$,
\item \label{it:Vg} 
$[V,\g] \ne [\g,\g]$.
\end{enumerate}
Conversely, given a subspace $V \subseteq \g$ of dimension 2 that satisfies \eqref{it:VV}, \eqref{it:VVg} and \eqref{it:Vg}, there exists $H \subseteq \g$ such that $(H,V)$ is an open respectful decomposition of $\g$.
\end{theorem}

As we show in Theorem~\ref{T:dim>=6}, no nilpotent Lie algebra of dimension at most 5 admits an open respectful decomposition. For algebras of dimension 6 we have the following result. 
 
\begin{theorem} \label{T:dimV=2}
Up to isomorphism, there are exactly 17 nilpotent Lie algebras of dimension 6 that admit an open respectful decomposition with $(H,V)$ with $\dim V =2, \, \dim H=4$.
\end{theorem}

Explicit details of the 17 algebras of Theorem~\ref{T:dimV=2} are given in Section~\ref{S:dimH4}; see Table~\ref{TableV2}.

If $(H,V)$ is open, then obviously $\dim H\ge 2$. In fact, we see in Lemma~\ref{L:JI}\eqref{L:H} that the case $\dim H= 2$ is impossible for open respectful decompositions.  Together with Theorem~\ref{T:dimV=2}, the following result gives a complete classification of nilpotent Lie algebras of dimension 6 admitting an open respectful decomposition.

\begin{theorem}\label{T:main}
Suppose that is a 6-dimensional nilpotent Lie algebra $\g$ possesses an open respectful decomposition $(H,V)$ with $\dim H=3$. Then the following conditions hold:
\begin{enumerate}
\item $\g$ doesn't have a codimension one Abelian ideal,
\item the center $Z(\g)$ of $\g$  has dimension $\le 2$,
\item the vector space sum $[\g,\g]+Z(\g)$ has dimension $\le 3$,
\item if $\dim Z(\g)=1$ and $\dim [\g,\g]=3$, then $[\g,[\g,[\g,\g]]]\not=0$.
\end{enumerate}
Up to isomorphism, there are exactly 14 nilpotent Lie algebras of dimension 6 satisfying the above four conditions. Conversely, each of these  14  algebras possesses an open respectful decomposition $(H,V)$ with $\dim H=3$.
\end{theorem}

Explicit details of the 14 algebras of Theorem~\ref{T:main} are given in Section~\ref{S:dimH3}; see Table~\ref{TableV3}.

\begin{definition}
We say that a decomposition $(H,V)$ of a Lie algebra $\g$ is \emph{mutually respectful} if $V$ respects $H$, and $H$ respects $V$.
\end{definition}

\begin{corollary}\label{cor} If a 6-dimensional nilpotent Lie algebra $\g$ has an open mutually respectful decomposition $(H,V)$, then either $\g\cong \h_3\oplus \h_3$ or $\g\cong \R\oplus \h_5$. 
\end{corollary}

\emph{Notation}. The Lie algebras in this paper are real and finite dimensional.
For a set $\{x_1,\dots,x_n\}$ of elements of a vector space, their linear span is denoted $\Span(x_1,\dots,x_n)$.
We will refer to the \emph{Heisenberg} algebras $\h_{2n+1}$, which have  basis $\{x_i,y_i,z:\ i=1,\dots,n\}$, and relations $[x_i,y_i]=z$ for $i=1,\dots,n$. 
We also use the 4-dimensional \emph{filiform} algebra
$\f_4$ which has  basis $\{x_1,\dots,x_4\}$, and relations $[x_1,x_2]=x_{3},\ [x_1,x_3]=x_{4}$.

\section{The case where $V$ respects $H$ and $V$ is a subalgebra}\label{S:tg}

In this section we consider the case where $(H,V)$ is a respectful decomposition of Lie algebra $\g$ and $V$ is a subalgebra of  $\g$.
The particular case where $\dim V=1$, say $V=\Span(v)$, has naturally received considerable attention. Here $v$ determines a geodesic vector field on the Lie group $G$ corresponding to $\g$, for any left-invariant Riemannian metric for which $H$ is orthogonal to $V$.
Such an element $v$ is also called a \emph{homogeneous geodesic} in the literature. Note that 
for such an element, the orbit through the identity element is a geodesic and at the same time a subgroup of $G$. Let us fix some terminology. Recall that a \emph{metric Lie algebra} is a Lie algebra 
$\g$ equipped with an inner product $\langle,\rangle$. 

\begin{definition} For a metric Lie algebra $(\g,\langle,\rangle)$, a non-trivial element $v$  of $\g$ is called a \emph{geodesic}  if the span $V:=\Span(v)$ of $v$ respects the orthogonal complement $\vp$ of $V$.
\end{definition}

One has the following basic result.

\begin{lemma}[\cite{CHGN}]\label{L:geo}
If $\g$ is a  Lie algebra and $Y\in \g$ is nonzero, then there is an inner product $\langle,\rangle$ on $\g$ for which $Y$ is a geodesic if and only if there does not exist $X\in\g$ with $[X,Y]=Y$.
\end{lemma}

 The existence of geodesics for a given inner product was proven by Ka{\u\i}zer; see \cite{Ka,KS,Du}.

\begin{theorem}[Ka{\u\i}zer]\label{T:somegeo}
Every metric Lie algebra possesses a geodesic. 
\end{theorem}

The existence of geodesics involves a number of interesting subtleties.
 It is easy to see that if a  Lie algebra $\g$ possesses an inner product with an orthonormal geodesic basis, then $\g$ is necessarily unimodular.
All semisimple Lie algebras have an inner product  with an orthonormal geodesic basis \cite{KS}, and so too do all nilpotent Lie algebras \cite{CLNN}. In \cite{CLNN}, it was proved  that for every unimodular Lie algebra of dimension $\leq 4$, every inner product  has an orthonormal geodesic basis, and an example was given  of  a 5-dimensional unimodular Lie algebra that has no orthonormal geodesic basis for any inner product; nevertheless this algebra does have a (nonorthonormal) geodesic basis for a certain inner product.

%\medskip
\begin{question}[\cite{CLNN}] Does every unimodular Lie algebra possess an inner product having a geodesic basis?
\end{question}

%\medskip
\begin{question}[\cite{CHNT}] Which unimodular Lie algebras possess an inner product having an orthonormal geodesic basis?
\end{question}

It is easy to find examples of nonunimodular Lie algebras that possess an inner product having a  geodesic basis, and others that don't.

%\medskip
\begin{question}[\cite{CHNT}] Which nonunimodular Lie algebras possess an inner product having a  geodesic basis?
\end{question}

Here are some of the results of \cite{CHNT}: 
\begin{itemize}
\item  Question 1 is resolved (in the affirmative) for Lie algebras having an Abelian derived algebra,
and for unimodular solvable Lie algebras with an Abelian nilradical, 
\item  Questions 1, 2 and 3 are  resolved for  Lie algebras having a codimension one Abelian ideal, and for  Lie algebras of dimension $2n$ having a codimension one ideal isomorphic to the Heisenberg Lie algebra
of dimension $2n-1$.
\item Question 3 is resolved for  Lie algebras of dimension $\le 4$.
\item Questions 1 and 2 are resolved for  Lie algebras of dimension $\le 5$.
\end{itemize}

That completes our discussion of the geodesic case. Now suppose $(H,V)$ is a respectful decomposition of Lie algebra $\g$ and $V$ is a subalgebra of  $\g$, and $\dim V> 1$. Remarkably little work has been done on this subject. Here is one known result.
Recall that a nilpotent Lie algebra $\g$ of dimension $n$ is said to be {\em filiform} if it possesses an element of maximal nilpotency; that is, there exists $X\in\g$  with $\ad^{n-2}(X)\not=0$, where $\ad(X):\g\to\g$ is the adjoint map $\ad(X)(Y)=[X,Y]$.

\begin{theorem}[{\cite[Theorem 1.16]{CHGN}}]\label{T:fili}
Suppose that $(H,V)$ is a respectful decomposition of a filiform Lie algebra $\g$ and that $V$ is a subalgebra. Then $\dim V\leq \lfloor\dim\g/2\rfloor$.
\end{theorem}

In fact, for some algebras, the bound $\dim V\leq \lfloor\dim\g/2\rfloor$ is not attained. The following example can be deduced from \cite[Theorem 1.19]{CHGN}.

\begin{example}\label{E:dim6biz}
For the $6-$dimensional filiform Lie algebra $\g$ with relations
\begin{align*}
[x_1,x_i]&=x_{i+1},\quad {\text{for\ }} i=2,\dots,5,\\
[x_2,x_3]&=-x_6,
\end{align*}
there is no respectful decomposition $(H,V)$ for which $V$ is a subalgebra and $\dim V=3$.
\end{example}

We remark that the respectful decompositions where $V$ is a subalgebra are a special case of a more general structure that has been studied. Suppose that a metric Lie algebra $(\g,\langle,\rangle)$ has a subalgebra $V$. Then $V$ is said to be \emph{totally geodesic} if 
$\langle[h,v_1],v_2\rangle +\langle v_1,[h,v_2]\rangle=0$ for all $v_1,v_2\in V,h\in \vp$. Obviously, $V$ is totally geodesic when $(\vp,V)$ is a respectful decomposition. For results on totally geodesic subalgebras, see \cite{CHGN,CHGN2}.

\section{Open respectful decompositions; basic facts}\label{S:nonint}

Suppose that $(H,V)$ is a respectful decomposition of a Lie algebra $\g$. As we remarked in the Introduction, the case where $H$ is a subalgebra is well understood. In the previous section we surveyed the case where $V$ is a subalgebra. In this section we  suppose that $(H,V)$ is open, i.e., neither $H$ nor $V$ is a subalgebra. 

\begin{lemma}\label{L:f4con}
 $\g$ doesn't have a codimension one Abelian ideal.
\end{lemma}

\begin{proof} Suppose that $\g$ has a codimension one Abelian ideal $\aa$. As $V$ is not a subalgebra, $V\not\subseteq \aa$, so there exists $v\in V$ with $v\not\in \aa$. Similarly, there exists $h\in H$ with $h\not\in \aa$. 
Then $h = a v + w$, where $a \ne 0$ and $w \in \mathfrak{a}$, and $H = \Span(h) \oplus H_0$, where $H_0 = H \cap \mathfrak{a}$. But then $H \supseteq [V, H] \supseteq [v, H_0] = [h, H_0] = [H, H]$, a contradiction.
\end{proof}

\begin{definition}
We introduce  three vector subspaces  of $\g$: 
\begin{align*}
H_V&=\Span \{\pi_H [v,v'] \ :\ v,v'\in V\},\\
V_H&=\Span \{\pi_V [h,h'] \ :\ h,h'\in H\},\\
K_H&= \{h\in H\ :\ \pi_V[h,h']=0,\ \text{for all }h'\in H\}.
\end{align*}
\end{definition}

\begin{remark}\label{R:KH} Obviously, if $H$ is not a subalgebra of $\g$, then $\dim K_H\le \dim H-2$.
\end{remark}

Note that the bracket 
$\pi_{V} [.,.]$ makes $V$ into a Lie algebra; indeed, the only point that needs verifying here is the Jacobi identity. If $v_1,v_2,v_3\in V$, then as $V$ respects $H$, one has
\[
\pi_V[v_1,\pi_V[v_2,v_3]]=\pi_V[v_1,[v_2,v_3]].
\]
Hence the Jacobi identity follows from that of $\g$.

\begin{definition}
We denote by $\V$ the Lie algebra structure on $V$ given by the bracket $\pi_{V} [.,.]$.
\end{definition}

Using only the Jacobi identity, one has:

\begin{lemma}\label{L:JI} \ 
\begin{enumerate}
\item\label{L:sa} $K_H$ is a subalgebra of $\g$, 
\item\label{L:VKH} $[V,K_H]\subseteq K_H$, 
\item\label{L:ker} $H_V\subseteq K_H$,
\item\label{L:VH1} $V_H$ is an ideal of $\V$,
\item\label{L:VH2} $\bH:=H+V_H$ is an ideal of $\g$,
\item\label{L:H}  $\dim H\ge 3$,
\item\label{L:KH} if $\dim H= 3$, then $H_V= K_H$ and $\dim H_V=1$.
\end{enumerate}
\end{lemma}

\begin{proof}(a).  Let $h_1,h_2\in K_H$.  Since $h_1\in K_H$, we have $[h_1,h_2]\in H$. Let $h'\in H$. Then $[h_1,h']\in H$ so $[h_2,[h_1,h']]\in H$. Similarly, $[h_1,[h_2,h']]\in H$. Thus the Jacobi identity gives $[[h_1,h_2],h']\in H$, as required.

(b).  Let $h\in K_H, v\in V$.  Since $V$ respects $H$,  we have $[h,v]\in H$. Let $h'\in H$. Since $h\in K_H$, we have $[h,h']\in H$ so $[v,[h,h']]\in H$. Similarly, $[h,[v,h']]\in H$. Thus the Jacobi identity gives $[[h,v],h']\in H$, as required.

(c). Suppose   $v,v'\in V$ and set $h=\pi_H[v,v']$.
Because $V$ respects $H$, for all $h'\in H$ the Jacobi identity gives
\begin{align*}
0&=\pi_V ([[v,v'],h']+[[h',v],v']+[[v',h'],v])=\pi_V\ [[v,v'],h']=\pi_V [\pi_H[v,v'],h']=\pi_V[h,h'],
\end{align*}
so $h\in K_H$. It follows that $H_V\subseteq K_H$.

(d). Let $h_1,h_2\in H$ and let $w:=\pi_V[h_1,h_2]$. Suppose that  $v\in V$.   As $V$ respects $H$, the Jacobi identity gives
\[
\pi_V[v,w]=\pi_V[v,[h_1,h_2]]=\pi_V[[v,h_1],h_2]+\pi_V[h_1,[v,h_2]]\in V_H.
\]
So (d) holds. (e) follows from (d) since $V$ respects $H$.

(f). As $H$ is not Lie subalgebra of $\g$, it has dimension at least two.  We assume that $\dim H=2$ and we will derive a contradiction. 
Since $\dim H=2$ and $H$ is not a Lie subalgebra of $\g$, we have $K_H=0$, by Remark~\ref{R:KH}. But this is impossible by (b), since $H_V\not=0$.

(g). As $V$ is not a subalgebra, we have $H_V\not=0$. As $\dim H= 3$ and $H$ is not a subalgebra, we have $\dim K_H\le 1$, by Remark~\ref{R:KH}. But $H_V\subseteq K_H$ by (b). So $H_V= K_H$. 
\end{proof}

\begin{example}\label{Ex:dim5nonu} Consider the non-unimodular solvable Lie algebra $\g$ of dimension 5 with basis $x_1,x_2,y_1,y_2,z$ and relations $[x_1,y_1]=z, [x_2,y_2]=z+y_2$.
Let $H=\Span(x_1,y_1,z+y_2)$ and $V=\Span(x_2,y_2)$. Clearly, 
$(H,V)$  is an open respectful decomposition of $\g$, the algebra $\V$ is Abelian, and $\dim [\g,\g]=2$.
\end{example}

\begin{example}\label{Ex:dim5} Consider the non-unimodular solvable Lie algebra $\g$ of dimension 5 with basis $v_1,v_2,h_1,h_2,h_3$ and relations:
\[
[v_1,v_2]=v_1+h_3,\quad [h_1,h_2]=v_1+h_3,\quad [h_1,v_2]=h_1.
\]
Let $H=\Span(h_1,h_2,h_3)$ and $V=\Span(v_1,v_2)$. Clearly, 
$(H,V)$  is an open respectful decomposition of $\g$,  the algebra $\V$ is not Abelian, and $\dim [\g,\g]=2$.
\end{example}

\begin{example}\label{Ex:dim52} Consider the non-unimodular solvable Lie algebra $\g$ of dimension 5 with basis $v_1,v_2,h_1,h_2,h_3$ and relations:
\[
[v_1,v_2]=-2v_1-2h_3,\quad [h_1,h_2]=v_1+h_3,\quad [v_2,h_1]=h_1,\quad [v_2,h_2]=h_2.
\]
Let $H=\Span(h_1,h_2,h_3)$ and $V=\Span(v_1,v_2)$. Clearly, 
$(H,V)$  is an open respectful decomposition of $\g$,   $\V$ is not Abelian, and $\dim [\g,\g]=3$.
\end{example}

\begin{problem} Classify all Lie algebras $\g$ of dimension 5 that have an open respectful decomposition.
\end{problem}

\begin{lemma}\label{L:derv} Suppose that a Lie algebra $\g$  has an open respectful decomposition $(H,V)$ and that $\dim V=2$. Then  $\dim [\g,\g]\ge 2$.
\end{lemma}

\begin{proof} Suppose that $\dim [\g,\g]=1$, with $[\g,\g]=\Span(z)$, say. Let $z=h+v$, where $h\in H$ and $v\in V$. Since neither $H$ nor $V$ is a Lie subalgebra of $\g$, we have $v$ and $h$ are both nonzero. In particular $H\cap  [\g,\g]=0$. Choose $h_1,h_2\in H$ such that $[h_1,h_2]=z$. Since  $\dim V=2$, we can choose $v'\in V$ such that $[v',v]=z$. For $i=1,2$, since $V$ respects $H$, we have $[v',h_i]\in H$, so $[v',h_i]=0$ since $[v',h_i]\in [\g,\g]$ and $H\cap  [\g,\g]=0$. Similarly, $[v',h]=0$. Therefore, the Jacobi identity gives
\[
0=[v',[h_1,h_2]]+[h_1,[h_2,v']]+[h_2,[v',h_1]]=[v',[h_1,h_2]]=[v',z]=[v',h+v]=[v',v]=z,
\]
which is a contradiction.
\end{proof}

\begin{example}\label{Ex:dim6} Consider the Lie algebra $\g$ of dimension 6 formed from the direct sum of $\R$ with the five dimensional Heisenberg algebra $\h_5$:   $\g$ has  basis $x_i,y_i,z,w$, for $i=1,2$, and relations $[x_i,y_i]=z$.  
 Let $H=\Span(x_1,y_1,z+w)$ and $V=\Span(x_2,y_2,z-w)$. Clearly, 
$(H,V)$ is an open respectful decomposition of $\g$. Note that in this case, 
the derived algebra $[\g,\g]=\Span(z)$ has dimension one.
\end{example}

\begin{lemma}\label{L:derv2} Suppose that a Lie algebra $\g$  has an open respectful decomposition $(H,V)$ and that $\dim H=3$. If $\dim \g=5,6$ or $8$, then  $\dim [\g,\g]\le \dim\g-1$.
\end{lemma}

\begin{proof} 
Arguing by contradiction, suppose that $[\g,\g]=\g$; i.e., $\g$ is perfect.
Since $\dim H= 3$ and neither $H$ nor $V$ is a Lie subalgebra, by Lemma~\ref{L:JI}\eqref{L:KH} we necessarily have $\dim H_V=1$ and $\dim V_H=1$. 
So $\dim(H+V_H)=4$. By Lemma~\ref{L:JI}\eqref{L:VH2}, $H+V_H$ is an ideal of $\g$. Consider the quotient algebra $\q:=\g/(H+V_H)$, which has dimension $\dim\g-4$. 
Since, $\g$ is perfect, $\q$ is perfect. But it is well known and easy to prove that there are no perfect Lie algebras of dimension 1, 2 or 4. This gives the required condition on $\dim \g$.
 \end{proof}

\begin{theorem}\label{T:derv} Suppose that a unimodular Lie algebra $\g$ of dimension five has an open respectful decomposition $(H,V)$. Then the derived algebra $[\g,\g]$ has dimension 3 or 4.
\end{theorem}

\begin{proof} From Lemma~\ref{L:derv2}, we have $\dim [\g,\g]\le 4$, and from Lemma~\ref{L:derv}, we have $\dim [\g,\g]\ge 2$. Suppose that $\dim [\g,\g]=2$. 
From Lemma~\ref{L:JI}\eqref{L:H}, we have $\dim H\ge 3$ and so as $\dim \g= 5$ and neither $H$ nor $V$ is a Lie subalgebra, $\dim H= 3$ and $\dim V=2$. 
As $V$ is not a Lie subalgebra and $\dim V=2$, we have that $\dim H_V=1$. 
Choose a basis $h_1,h_2,h_3$  for $H$ so that $H_V =\Span(h_3)$. So by Lemma~\ref{L:JI}\eqref{L:ker}, $V_H=\Span(\pi_V[h_1,h_2])$. In particular, $\dim V_H=1$. 

First suppose that  $\V$ is not Abelian. So the derived algebra $\pi_V[\V,\V]$ has dimension one. Thus, as  $\dim V_H=1$ and by  Lemma \ref{L:JI}\eqref{L:VH1}, $V_H$ is an ideal of $\V$, we have $V_H=\pi_V[\V,\V]$. Choose basis elements $v_1,v_2$ for $V$ such that 
$\pi_V[h_1,h_2]=v_2$ and $[v_1,v_2]=v_2+h_3$.

Since $[\g,\g]\subseteq \Span(h_1,h_2,h_3,v_2)$ and $\dim [\g,\g]=2$, we have $[\g,\g]=\Span(v_2+h_3,h)$, for some $h\in H$.
So $H\cap [\g,\g]=\Span(h)$.
Thus, as $V$ respects $H$,  we have $[v_1,h_i]\in \Span(h)$ for $i=1,2,3$.
The Jacobi identity gives
\begin{equation}\label{E:vhht3b}
[[v_1,h_1],h_2]+[h_1,[v_1,h_2]]=[v_1,[h_1,h_2]].
\end{equation}
If $h\in \Span(h_3)$, then
 the LHS of \eqref{E:vhht3b} belongs to $H$ by Lemma~\ref{L:JI}\eqref{L:ker}, while for the RHS we have
$\pi_V[v_1,[h_1,h_2]]=\pi_V[v_1,v_2]=v_2$, giving a contradiction. So $h\not\in \Span(h_3)$.

We have $[v_1,h]=rh$ for some $r\in \R$. Choose $h'\in H$ such that $h,h',h_3$ is a basis for $H$, and such that $\pi_V[h',h]=v_2$.  The Jacobi identity gives
\begin{equation}\label{E:vhht3c}
[[v_1,h'],h]+[h',[v_1,h]]=[v_1,[h',h]].
\end{equation}
As  $[v_1,h']\in\Span(h)$, the $V$ component of the LHS of \eqref{E:vhht3c} is $\pi_V[h',[v_1,h]]=\pi_V[h',rh]=rv_2$.
The $V$ component of the RHS of \eqref{E:vhht3c} is $\pi_V[v_1,v_2]=v_2$. So $r=1$.
From the above considerations,  we have 
\[
[v_1,v_2]=v_2+h_3,\quad [v_1,h]=h,\quad [v_1,h']\in \Span(h),\quad [v_1,h_3]\in \Span(h), 
\]
so $\ad(v_1)$ has trace 2, contradicting the hypothesis that $\g$ is unimodular. We conclude that $\V$ is Abelian.

Choose basis elements $v_1,v_2$ for $V$ so that $[v_1,v_2]=h_3$ and $[h_1,h_2]=v_2+h$, for some $h\in H$.
We have $[\g,\g]=\Span(v_2+h,h_3)$.
In particular, $H\cap [\g,\g]=\Span(h_3)$.
Thus, as $V$ respects $H$,  we have $[v,h_i]\in \Span(h_3)$ for all $v\in V$ and $i=1,2,3$.
Since $\g$ is unimodular and $\V$ is Abelian, it follows that $[v,h_3]=0$ for all $v\in V$.

For all $f\in H$, the Jacobi identity gives
\begin{equation}\label{E:fvv}
[[f,v_1],v_2]+[v_1,[f,v_2]]=[f,[v_1,v_2]].
\end{equation}
The LHS is 0, since $[f,v_i]\in \Span(h_3)$ for $i=1,2$, and $[v,h_3]=0$ for all $v\in V$.
The RHS is $[f,h_3]$. Hence $[f,h_3]=0$ for all $f\in H$.

The Jacobi identity also gives
\begin{equation}\label{E:vhht3e}
[[v_1,h_1],h_2]+[h_1,[v_1,h_2]]=[v_1,[h_1,h_2]].
\end{equation}
Since $[v_1,h_i]\in \Span(h_3)$ for $i=1,2$ and $[f,h_3]=0$  all  $f\in H$, the LHS of \eqref{E:vhht3e} is $0$. The RHS is
$[v_1,[h_1,h_2]]=[v_1,v_2+h]=h_3+[v_1,h]$. So $[v_1,h]=-h_3$. In particular, since $[v_1,h_3]=0$, we have $h\not\in \Span(h_3)$. Choose $h'\in H$ such that $h,h',h_3$ is a basis for $H$, and furthermore, that $[h',h]=[h_1,h_2]=v_2+h$. 
From the above considerations,  we have 
\[
[h',v_1]\in \Span(h_3),\quad [h',v_2]\in \Span(h_3),\quad [h',h]=v_2+h
,\quad [h',h_3]=0, 
\]
so $\ad(h')$ has trace 1, contradicting the hypothesis that $\g$ is unimodular. This conclude the proof.
\end{proof}

\begin{remark} Note that $\dim [\g,\g]=2$ can occur in the non-unimodular case; see Examples~\ref{Ex:dim5nonu}  and \ref{Ex:dim5}.
\end{remark}

We will show below in Theorem~\ref{T:dim>=6}  that not all unimodular Lie algebras of dimension five having a derived algebra of dimension 3 or 4, possess an open respectful decomposition.
First note that even though in general $\V$ is not a subalgebra of $\g$, and the map $\pi_V:\g\to \V$ is not a Lie algebra homomorphism, we have:

\begin{lemma}\label{L:nil} If $\g$ is nilpotent, then $\V$ is nilpotent. 
\end{lemma}

\begin{proof}  If $\g$ is nilpotent, then for all $v\in V$ the map $\ad(v)$ is nilpotent  and hence the map $\pi_V\circ\ad(v)$ is nilpotent. So $\g$ is nilpotent by Engel's theorem.
 \end{proof}

\begin{example}\label{Ex:h3x2} Consider the Lie algebra $\g=\h_3\oplus\h_3$ of dimension 6 formed from the direct sum of two copies of the 3-dimensional Heisenberg algebra $\h_3$:   $\g$ has  basis $x_i,y_i,z_i$, for $i=1,2$, and relations $[x_i,y_i]=z_i$.  We give two examples:

(a). Let $H=\Span(x_1,y_1,z_2)$ and $V=\Span(x_2,y_2,z_1)$. Clearly, 
$(H,V)$ is an open respectful decomposition of $\g$. In this case, 
the Lie algebra $\V$ is Abelian.

(b). Let $H=\Span(x_1,y_1,z_1+z_2)$ and $V=\Span(x_2,y_2,z_1-z_2)$. Clearly, 
$(H,V)$ is an open respectful decomposition of $\g$. In this case, 
$\V$ is isomorphic to the 3-dimensional Heisenberg algebra $\h_3$.
\end{example}

%%%%%%%%%%%%%%%%%%%%%%%%%%%%%%%%%%%%%%%%%%%%%%%%%%%%

\section{Open respectful decompositions with $\g$ nilpotent and $\dim V=2$}\label{S:dimH4}

\begin{proof}[Proof of Theorem \ref{T:v2}]
  Suppose $(H,V)$ is an open respectful decomposition of $\g$ with $\dim V = 2$. Assertion~\eqref{it:VV} follows by definition. Moreover, from Lemma~\ref{L:nil} we have $[V,V] \subseteq H$. Then $[V,\g] \subseteq H$ which implies assertion~\eqref{it:VVg}. If $[V,\g] = [\g,\g]$, then $[\g,\g] \subseteq H$, and so $[H,H] \subseteq H$. This contradiction proves assertion~\eqref{it:Vg}. 
  
  Conversely, suppose $V \subseteq \g$ is a subspace of dimension 2 that satisfies \eqref{it:VV}, \eqref{it:VVg} and \eqref{it:Vg}. If $V$ were a subalgebra, then by \eqref{it:VVg} we would have had $[V,V]=0$ contradicting~\eqref{it:VV}. To construct $H$, consider the quotient linear space $W=\g/[V,\g]$. Let $\pi:\g \to W$ be the natural projection, and denote $m = \dim W$. Note that the codimension of $[\g,\g]$ in $\g$ is at least $2$, as $\g$ is nilpotent. As $[V,\g]$ is a proper subspace of $[\g,\g]$, we get $m \ge 3$. The subspace $\pi(V) \subseteq W$ has dimension $2$ by \eqref{it:VVg}, and so the set $S_1$ of $(m-2)$-dimensional subspaces which complement $\pi(V)$ to $W$ is an open and dense subset in the Grassmannian $G(m-2,W)$. Moreover, by \eqref{it:Vg} we have $\pi [\g,\g] \ne 0$, and so the set $S_2$ of $(m-2)$-dimensional subspaces which do not contain $\pi [\g,\g]$ is also open and dense in $G(m-2,W)$. Take any element $H' \in S_1 \cap S_2$ and define $H = \pi^{-1}(H')$. We have $H \oplus V = \g$, as $\pi(V) \oplus H' = W$ and the restriction of $\pi$ to $V$ is injective. Furthermore, $[V,H] \subseteq [V,\g] =\pi^{-1}(0) \subseteq H$, and $\pi [H,H] \subseteq \pi [\g,\g] \not\subseteq H'=\pi(H)$, as $H' \in S_2$. This implies that $[H,H] \not\subseteq H$ which completes the proof.
\end{proof} 
We remark that any linear subspace that complements $V$ in $\g$, contains $[V,\g]$ and does not contain $[\g,\g]$ can serve as $H$.

As an application of Theorem~\ref{T:v2}, we prove the following proposition which covers a fairly large class of nilpotent algebras.

\begin{proposition} \label{P:derna}
  Suppose the derived algebra $[\g,\g]$ of a nilpotent algebra $\g$ is non-Abelian, and that $\dim [\g,\g] - \dim [\g,[\g,\g]] \ge 2$. Then $\g$ admits an open respectful decomposition $(H,V)$ with $\dim V = 2$.
\end{proposition}
Note that $\dim [\g,\g] - \dim [\g,[\g,\g]] \ge 1$ for any nilpotent, non-Abelian Lie algebra $\g$.
\begin{proof}
  Choose two vectors $v_1, v_2 \in [\g,\g]$ which are linearly independent modulo $[\g,[\g,\g]]$ and define $V=\Span(v_1, v_2)$. The set of so constructed 2-dimensional subspaces $V \subseteq [\g,\g]$ is open and dense in the Grassmanian $G(2,[\g,\g])$, and so if $[V,V]=0$ for all of them, then $[\g,\g]$ is Abelian contradicting the assumption. Therefore without loss of generality we can assume that $[V,V] \ne 0$, and so $V$ satisfies condition~\eqref{it:VV} of Theorem~\ref{T:v2}. Moreover, as $V \subseteq [\g,\g]$, we have $[V,\g] \subseteq [\g,[\g,\g]]$, which shows that $V$ also satisfies conditions~\eqref{it:VVg} and \eqref{it:Vg} of Theorem~\ref{T:v2}.
\end{proof}
A similar construction works if we replace $[\g,\g]$ by any other term of the central descending series of $\g$ (except for $\g$ itself), and $[\g,[\g,\g]]$, by the subsequent term. %We note that although the assumptions in Proposition~\ref{P:derna} are relatively mild in higher dimensions, the algebras $\g$ which satisfy them must be at least 4-step nilpotent.  

\begin{remark}\label{R:nonV2} 
Examples of Lie algebras admitting no open respectful decomposition with $(H,V)$ with $\dim V =2$ include algebras whose derived algebra is 1-dimensional (by Lemma~\ref{L:derv}) and algebras having a codimension 1 Abelian ideal (by  Lemma~\ref{L:f4con}). There are many other 2-step nilpotent Lie algebras having no open respectful decomposition $(H,V)$ with $\dim V =2$, for example, the nonsingular 2-step nilpotent Lie algebras \cite{Ebe}.
\end{remark}

% include factor by ideal; include direct sum (helps with 4 + 2)?
% anything to say about 2-step?
% 6 22 unique up to \pm and adding the centre V; depends on the field; and 6 24 too

\begin{table}
\begin{tabular}{ c | c | c | c | c }
  \hline
Algebra & Relations & $[\g,\g]$ basis & $V$ basis & $[V,\g]$ basis \\ \hline
$L_{6,10}$ & $[x_1,x_2]=x_3,[x_1,x_3]=x_6,[x_4,x_5]=x_6$ & $x_3,x_6$ & $x_4,x_5$ & $x_6$ \\\hline
$L_{6,11}$ & $[x_1,x_2]=x_3,[x_1,x_3]=x_4,[x_1,x_4]=x_6,$ & $x_3,x_4,x_6$ & $x_2,x_5$ & $x_3,x_6$ \\
 & $[x_2,x_3]=x_6,[x_2,x_5]=x_6$ &  &  &   \\ \hline
$L_{6,12}$ & $[x_1,x_2]=x_3,[x_1,x_3]=x_4,$ & $x_3,x_4,x_6$ & $x_2,x_5$ & $x_3,x_6$  \\
 & $[x_1,x_4]=x_6,[x_2,x_5]=x_6$ &  &  &     \\ \hline
$L_{6,13}$ & $[x_1,x_2]=x_3,[x_1,x_3]=x_5,[x_2,x_4]=x_5,$ & $x_3,x_5,x_6$ & $x_3,x_4$ & $x_5,x_6$  \\
 & $[x_1,x_5]=x_6,[x_3,x_4]=x_6$ &  &  &     \\ \hline
$L_{6,14}$ & $[x_1,x_2]=x_3, [x_1,x_3]=x_4, [x_1,x_4]=x_5,$ & $x_3,x_4,x_5,x_6$ & $x_2+x_3,x_4$ & $x_3+x_4,x_5,x_6$  \\
 & $[x_2,x_3]=x_5, [x_2,x_5]=x_6,[x_3,x_4]=-x_6$ &  &  &     \\ \hline
$L_{6,15}$ & $[x_1,x_2]=x_3, [x_1,x_3]=x_4, [x_1,x_4]=x_5,$ & $x_3,x_4,x_5,x_6$ & $x_2,x_4$ & $x_3,x_5,x_6$  \\
 & $[x_2,x_3]=x_5,[x_1,x_5]=x_6,[x_2,x_4]=x_6$ &  &  &    \\ \hline
$L_{6,16}$ & $[x_1,x_2]=x_3, [x_1,x_3]=x_4, [x_1,x_4]=x_5,$ & $x_3,x_4,x_5,x_6$ & $x_2,x_5$ & $x_3,x_6$  \\
 & $[x_2,x_5]=x_6, [x_3,x_4]=-x_6$ &  &  &     \\ \hline
$L_{6,17}$ & $[x_1,x_2]=x_3, [x_1,x_3]=x_4, [x_1,x_4]=x_5,$ & $x_3,x_4,x_5,x_6$ & $x_2+x_4,x_3$ & $x_3+x_5,x_4,x_6$  \\
 & $[x_1,x_5]=x_6, [x_2,x_3]= x_6$ &  &  &     \\ \hline
$L_{6,19}(\epsilon)$ & $[x_1,x_2]=x_4,[x_1,x_3]=x_5,$ & $x_4,x_5,x_6$ & $x_2,x_4+x_5$ & $x_4,x_6$  \\
$\epsilon=0, \pm 1$ & $[x_2,x_4]=x_6, [x_3,x_5]=\epsilon x_6$ &  &  &    \\ \hline
$L_{6,20}$ & $[x_1,x_2]=x_4,[x_1,x_3]=x_5,$ & $x_4,x_5,x_6$ & $x_2,x_4+x_5$ & $x_4,x_6$  \\
 & $[x_1,x_5]=x_6, [x_2,x_4]= x_6$ &  &  &   \\ \hline
$L_{6,21}(\epsilon)$ & $[x_1,x_2]=x_3, [x_1,x_3]=x_4, [x_2,x_3]=x_5,$ & $x_3,x_4,x_5,x_6$ & $x_1,x_4+x_5$ & $x_3,x_4,x_6$  \\
$\epsilon=0, \pm 1$  & $[x_1,x_4]=x_6, [x_2,x_5]= \epsilon x_6$ &  &  &    \\ \hline
$L_{6,22}(1)$ & $[x_1,x_2]=x_5,  [x_1,x_3]=x_6,$ & $x_5,x_6$ & $x_1 + x_4,x_2- x_3$ & $x_5 - x_6$   \\
 & $ [x_2,x_4]= x_6,  [x_3,x_4]=x_5$ &  &  &   \\\hline
$L_{6,24}(1)$  & $[x_1,x_2]=x_3,[x_1,x_3]=x_5,[x_1,x_4]= x_6,$ & $x_3,x_5,x_6$ & $x_1+ x_2,x_3 + x_4$ & $x_3,x_5+x_6$ \\
  & $[x_2,x_3]=x_6,[x_2,x_4]=x_5$ &  &  &    \\  \hline
%
%$L_{6,22}(\epsilon)$ & $[x_1,x_2]=x_5,  [x_1,x_3]=x_6,$ & $x_5,x_6$ & $\alpha x_1 + x_4,x_2-\alpha x_3$ & $x_5 - \alpha x_6$   \\
%$\epsilon = \alpha^2 >0$ & $ [x_2,x_4]=\epsilon x_6,  [x_3,x_4]=x_5$ &  &  &   \\\hline
%
%$L_{6,24}(\epsilon)$  & $[x_1,x_2]=x_3,[x_1,x_3]=x_5,[x_1,x_4]=\epsilon x_6,$ & $x_3,x_5,x_6$ & $x_1+ \alpha x_2,\alpha x_3+ x_4$ & $x_3,x_5+\alpha x_6$ \\
%$\epsilon = \alpha^2 >0$  & $[x_2,x_3]=x_6,[x_2,x_4]=x_5$ &  &  &  &   \\  \hline
\end{tabular}
\bigskip
\caption{The 17 nilpotent 6-dimensional algebras admitting an open respectful decomposition $(H,V)$ with $\dim V = 2$}\label{TableV2}
\end{table}

\begin{proof}[Proof of Theorem \ref{T:dimV=2}]
 See Table~\ref{TableV2}, which uses the classification and notation of de Graaf \cite{deG}. For each of these 17 algebras,  Table~\ref{TableV2} gives the subspace $V$; one can take any linear subspace which complements $V$ in $\g$, contains $[V,\g]$ and does not contain $[\g,\g]$ as $H$.
  It is easy to verify that for every algebra $\g$ listed in Table~\ref{TableV2}, the corresponding subspace $V$ satisfies all three conditions in Theorem~\ref{T:v2}.  To see that the other algebras in de Graaf's classification do not possess an open respectful decomposition, we first note that the algebras $\R^2 \oplus L_{4,1}, \, \R^2 \oplus L_{4,2}$ and $\R \oplus L_{5,4}$ have derived algebras of dimension at most 1, and the algebras $\R^2 \oplus L_{4,3}, \, \R \oplus L_{5,7}, \, \R \oplus L_{5,8}, \, L_{6,18}$ and $L_{6,25}$ have codimension 1 Abelian ideal.
  
  In the remaining cases (for algebras $\R \oplus L_{5,5}, \, \R \oplus L_{5,6}, \, \R \oplus L_{5,9}, \, L_{6,22}(0), \, L_{6,22}(-1)$, $L_{6,23}, \, L_{6,24}(0)$, $L_{6,24}(-1)$ and $L_{6,26}$), one can verify that conditions~\eqref{it:VV} and \eqref{it:Vg} of Theorem~\ref{T:v2} cannot simultaneously be satisfied. To facilitate the calculation, we note that it suffices to check this for 2-dimensional subspaces $V$ lying in the span of the non-central coordinate vectors $x_i$ in de Graaf's presentations.
  
  The calculations for all the algebras follow the same scheme, which we demonstrate for  the two most involved cases, for the algebras $L_{6,24}(\epsilon), \, \epsilon = 0, -1$. From the defining relations for these algebras (given in the last row of Table~\ref{TableV3}) we see that $[\g,\g]=\Span(x_3,x_5,x_6)$, and that $Z(\g)=\Span(x_5,x_6)$. For a vector $v = \sum_{i=1}^{6} \mu_i x_i$ we obtain $[v,\g] = \Span(\mu_2 x_3 +\mu_3 x_5 + \epsilon \mu_4 x_6, \mu_1 x_3 -\mu_4 x_5 - \mu_3 x_6, \mu_1 x_5 + \mu_2 x_6, \mu_2 x_5 + \epsilon \mu_1 x_6)$. For $[v,\g]$ not to 
   contain $[\g,\g]$, we must have $\mu_2^2 - \epsilon \mu_1^2=0$. If $\epsilon = -1$, then any such $v$ lies in $\Span(x_3,x_4,x_5,x_6)$, which is an Abelian ideal, so for any 2-dimensional subspace of $L_{6,24}(-1)$, either condition~\eqref{it:VV} or condition \eqref{it:Vg} of Theorem~\ref{T:v2} is violated. If $\epsilon = 0$, we have $\mu_2=0$. Two vectors $v_1= \sum_{i=1}^{6} \mu_i x_i, v_2= \sum_{i=1}^{6} \nu_i x_i$ with $\mu_2=\nu_2=0$ do not commute when $\mu_1\nu_3 - \mu_3 \nu_1 \ne 0$. But then $[v_1,\g] + [v_2,\g] = [\g,\g]$, so no 2-dimensional subspace $V \subseteq L_{6,24}(0)$ may satisfy both condition~\eqref{it:VV} and condition~\eqref{it:Vg} of Theorem~\ref{T:v2}.
\end{proof}

It is curious to observe the sensitivity of the existence of the decomposition to the ground field: the algebras $L_{6,22}(-1)$ and $L_{6,24}(-1)$ over $\mathbb{C}$ are isomorphic to the algebras $L_{6,22}(1)$ and $L_{6,24}(1)$, respectively, and as such, admit open respectful decompositions $(H,V)$ with $\dim V = 2$.

%%%%%%%%%%%%%%%%%%%%%%%%%%%%%%%%%%%%%%%%%%%

\section{Open respectful decompositions with $\g$ nilpotent and $\dim H=3$}\label{S:dimH3}

In this section we will suppose that  $\g$ is a nilpotent Lie algebra  having an open respectful decomposition $(H,V)$, and furthermore that $H$ has dimension 3. We have $\dim H_V=1$ by Lemma~\ref{L:JI}\eqref{L:KH}. Choose $h_3\in H$  so that $H_V =\Span(h_3)$. 

\begin{lemma}\label{L:VHV}
 If $\g$ is nilpotent and $H$ has dimension 3,  then  $V$ commutes with $H_V$.
So the subspace $\bV:=V+H_V$ is a subalgebra of $\g$ and $H_V$ is contained in the center of $\bV$.
\end{lemma}

\begin{proof} Let $h\in H, v\in V$. Since $K_H =H_V=\Span(h_3)$, we have  $[h_3,h]\in H$, so $[v,[h_3,h]]\in H$. Similarly, $[h,v]\in H$ so $[h_3,[h,v]]\in H$. Hence, by the Jacobi identity,
\[
0=\pi_V([h,[v,h_3]]+[h_3,[h,v]]+[v,[h_3,h]])=\pi_V[h,[v,h_3]],
\]
so $[v,h_3]\in K_H$. So $[v,h_3]$ is a multiple of $h_3$. Hence as  $\g$ is nilpotent, $v$ commutes with $h_3$. The required result follows.
\end{proof}

Since $h_3\in H_V=K_H$, we have $[h_3,h]\in H$ for all $h\in H$.
Since $V$ respects $H$, we have $[v,h]\in H$ for all $v\in V,h\in H$.
Hence the adjoint  action of  $\bV$ respects $H$. Furthermore,  by Lemma~\ref{L:VHV}, we have $h_3\in Z(\bV)$; i.e., the adjoint  action of  $\bV$ on $H$ sends $h_3$ to $0$.
Let $\underline H$ denote the quotient vector space $H/\Span(h_3)$. The adjoint  action of  $\bV$ induces an action of $\bV$ of $\underline H$. Since $\bV$ is nilpotent, there is a basis $\{\underline h_1,\underline h_2\}$ for $\underline H$ on which the action is lower triangular. In other words, there is  a basis  $\{h_1,h_2,h_3\}$ for $H$ for which 
\begin{equation}\label{E:tria}
[z,h_1]\in \Span(h_2,h_3),\qquad [z,h_2]\in \Span(h_3),\qquad [z,h_3] =0,
\end{equation}
for all $z\in \bV$. In particular, as $h_3\in  \bV$, we have $[h_3,h_2]\in \Span(h_3)$, and so, as $\bH$ is nilpotent, $[h_3,h_2]=0$. 
Notice also that as $h_3\in H_V$ and $H$ is not a subalgebra, we have $[h_1,h_2]\not\in H$. We will use the basis $\{h_1,h_2,h_3\}$ for $H$   for the rest of this paper.

\begin{theorem}\label{T:dim>=6} If $\g$ is nilpotent Lie algebra and has an open respectful decomposition $(H,V)$,  then $\dim \g\ge 6$.
\end{theorem}

\begin{proof} From Lemma~\ref{L:JI}\eqref{L:H}, we have $\dim H\ge 3$ and so $\dim \g\ge 5$. Suppose that $\dim \g=5$ and $\dim H=3$. 
As  $\V$ is Abelian by Lemma~\ref{L:nil}, and $H_V =\Span(h_3)$, we may choose a basis $\{v_1,v_2\}$ for $V$ so that $[v_1,v_2]=h_3$ and 
$\pi_V[h_1,h_2]=v_1$.
Thus $\bH=\Span(h_1,h_2,h_3,v_1)$. Since $\g$  is nilpotent, its center $Z(\g)$ is nonzero. As   
$\pi_V[h_1,h_2]\not=0$ and $h_3\in K_H$, and as $V$ respects $H$, it follows that  $Z(\g)\subseteq \Span(v_1,v_2,h_3)$. As $[v_1,v_2]=h_3$, one has $Z(\g)=\Span(h_3)$. 
Indeed, if  $av_1+bv_2+ch_3$ is a nonzero element of $Z(\g)$, for some constants $a,b,c$, then obviously $c\not=0$, so let us take $c=1$. Then $0=[av_1+bv_2+h_3,v_2]=ah_3+[h_3,v_2]=ah_3$, by \eqref{E:tria}, and hence $a=0$. Similarly, $b=0$. So $Z(\g)=\Span(h_3)$.

Let $[h_1,h_2]=v_1+dh_1+eh_2+fh_3$ for some $d,e,f\in\R$. 
Since $h_3\in Z(\g)$, we have  $[h_1,h_3]=0$. From \eqref{E:tria} we have $[h_1,v_1]\in \Span(h_2,h_3)$. 
Thus $\tr(\ad(h_1))=e$. So, as $\bH$ is nilpotent and hence unimodular, $e=0$.
Similarly,  $[h_2,h_3]=0$ and from \eqref{E:tria} we have $[h_2,v_1]\in \Span(h_3)$. 
Hence $\tr(\ad(h_2))=d$ and so $d=0$. Thus $[h_1,h_2]=v_1+fh_3$.
Consequently, as $h_3\in Z(\g)$,
\begin{equation}\label{E:fix1}
[[h_1,h_2],v_2]=[v_1+fh_3,v_2]=[v_1,v_2]=h_3.
\end{equation}
The Jacobi identity gives
\begin{equation}\label{E:fix2}
[[h_1,h_2],v_2]=[[h_1,v_2],h_2]+[h_1,[h_2,v_2]].
\end{equation}
But from \eqref{E:tria}, we have
$[h_1,v_2]\in \Span(h_2,h_3)$ and $ [h_2,v_2]\in \Span(h_3)$. So, as $[h_3,h_1]=[h_3,h_2]=0$, we obtain $[[h_1,h_2],v_2]=0$. But then \eqref{E:fix1} gives $h_3=0$, which is a contradiction.
\end{proof}

\begin{lemma}\label{L:cen}
 If $\V$ is nilpotent and $H$ has dimension 3,  then  $V_H\subseteq Z(\bV)$.
\end{lemma}

\begin{proof} 
Let $w=\pi_V[h_1,h_2]$, so $V_H=\Span(w)$. We will show that $w\in Z(\bV)$.
Suppose that  $v\in V$.   
Using the notation introduced at the beginning of this section, we have $[v,h_1]=ah_2+bh_3$ and $[v,h_2]=ch_3$, for some $a,b,c\in \R$.
 As $V$ respects $H$, the Jacobi identity gives
\[
\pi_V[v,w]=\pi_V[v,[h_1,h_2]]=\pi_V[[v,h_1],h_2]+\pi_V[h_1,[v,h_2]]=\pi_V[bh_3,h_2]+\pi_V[h_1,ch_3]=0,
\]
since $h_3\in K_H$. Thus $[v,w]=\lambda h_3$, for some $\lambda\in \R$. Suppose $\lambda\not=0$.
Thus $v\not\in \Span(w)$. Let $\g':=\Span(v,w,h_1,h_2,h_3)$. Note that $\g'$ is a 5-dimensional subalgebra of $\g$. In particular, $\g'$ is  nilpotent.  Set $V'=\Span(v,w)$. Then $(H,V')$ is  an open respectful decomposition of $\g'$. But this is impossible by Theorem~\ref{T:dim>=6}. Hence $\lambda=0$ and so $[v,w]=0$. Since $v$ was arbitrary, we have $w\in Z(\bV)$, as required.
\end{proof}

%remark: above holds when $\g$ is nilpotent and $K_H$ has dim one. eg, when $\dim V=2$ and $V_H=V$.

If $\g$ is nilpotent with  $\dim \g=6$ and $\dim H=3$, then  $\bV$ is a 4-dimensional non-Abelian nilpotent subalgebra of $\g$. So  $\bV$ is isomorphic to either $\R\oplus \h_3$ or the 4-dimensional filiform algebra $\f_4$. In fact, the second possibility does not occur.

\begin{lemma}\label{L:VHV6}
Suppose that $\g$ is nilpotent and that $\dim \g=6$ and $\dim H=3$. Then  $\bV$ is isomorphic to $\R\oplus \h_3$. 
\end{lemma}

\begin{proof} 
Note that the center of $\R\oplus \h_3$ has dimension two, while
the center of $\f_4$ has dimension one.
Using the notation introduced at the beginning of this section, we have a basis $\{h_1,h_2,h_3\}$ for $H$ for which $\Span(h_3)=H_V$ and $\Span(\pi_V[h_1,h_2])=V_H$. 
By Lemma~\ref{L:VHV}, $h_3\in Z(\bV)$, and by Lemma~\ref{L:cen}, $\pi_V[h_1,h_2]\in Z(\bV)$. So $\dim Z(\bV)>1$ and hence $\bV\cong\R\oplus \h_3$.
\end{proof}

Suppose that $\dim\g=6$ and $\dim H=3$. Using the notation introduced at the beginning of this section,  we have a basis $\{h_1,h_2,h_3\}$ for $H$ for which 
\begin{equation}\label{E:tria2}
[h_2,h_3] =0,\qquad \pi_V[h_1,h_2]\not=0,\qquad [z,h_1]\in \Span(h_2,h_3),\qquad [z,h_2]\in \Span(h_3),\qquad [z,h_3] =0,
\end{equation}
for all $z\in \bV$.
Choose a  basis  $\{v_1,v_2,v_3\}$ for $\V$ such that 
$\pi_V[h_1,h_2]=v_3$. Since 
$V_H\subseteq Z(\bV)$ by Lemma~\ref{L:cen}, we have $[v_1,v_3]=[v_2,v_3]=0$. So, as $V$ is not a subalgebra, $\pi_H[v_1,v_2]\not=0$. So, by rescaling the elements $v_1,v_2$ if necessary, we may assume that $\pi_H[v_1,v_2]=h_3$. Notice that $\V$ is either Abelian or isomorphic to $\h_3$.
So, as 
$V_H\subseteq Z(\bV)$,  by rescaling the elements $v_1,h_3$ if necessary we may assume $[v_1,v_2]=h_3+\delta v_3$, where  $\delta=1$ when $\V$ is isomorphic to $\h_3$, and $\delta=0$ when $\V$ is Abelian.
By the previous lemma, $\bV\cong \R\oplus \h_3$, so its center has dimension two. From Lemma~\ref{L:cen},  $v_3\in Z(\bV)$ and by Lemma~\ref{L:VHV}, $h_3\in Z(\bV)$.
So $Z(\bV)=\Span(h_3,v_3)$.
Hence, in addition to \eqref{E:tria2}, we have the relations
\begin{equation}\label{E:vs}
[v_1,v_2]=h_3+\delta v_3, \qquad Z(\bV)=\Span(h_3,v_3), \qquad\pi_V[h_1,h_2]=v_3.
\end{equation}

The algebra $\V$ is either (A) Abelian or (B) isomorphic to $\h_3$. The algebra $\bH$ is four dimensional, nilpotent and non-Abelian. So either (I) $\bH\cong \f_4$ or (II) $\bH\cong\R\oplus \h_3$. So we have four cases to consider. 

\begin{lemma}\label{L:ceng}
Suppose that $\g$ is nilpotent and that $\dim \g=6$ and $\dim H=3$. Then  $Z(\g)\subseteq \Span(h_3,v_3)$.
\end{lemma}

\begin{proof}
Let $z:=ah_1+bh_2+ch_3+dv_1+ev_2+fv_3\in Z(\g)$. 
Then as $V$ respects $H$ and $[h_i,h_3]\in H$ for $i=1,2$, we have that $\pi_V[h_1,z]=0$ gives $b=0$ and $\pi_V[h_2,z]=0$ gives $a=0$.
So $0=[v_1,z]=[v_1,ev_2+fv_3]=e(h_3+\delta v_3)$, and so $e=0$. Similarly $0=[v_1,z]$ gives $d=0$. The required result follows.
\end{proof}

\begin{lemma}\label{L:Vab}
Suppose that $\g$ is nilpotent and that $\dim \g=6$ and $\dim H=3$. If $\V$ is Abelian, then  $h_3\in Z(\g)$.
\end{lemma}

\begin{proof}
Suppose that $\V$ is Abelian.
Applying the Jacobi identity gives
\begin{equation}\label{E:h3}
[h_1,h_3]=[h_1,[v_1,v_2]]=[[h_1,v_1],v_2]+[v_1,[h_1,v_2]].
\end{equation}
From \eqref{E:tria2} we have $[[h_1,v_1],v_2],[v_1,[h_1,v_2]]\in \Span(h_3)$, so \eqref{E:h3}
gives $[h_1,h_3]\in \Span(h_3)$. Thus, as $\bH$ is nilpotent, $[h_1,h_3]=0$. So, since $h_3\in Z(\bV)$ by Lemma~\ref{L:VHV},  and $[h_2,h_3] =0$ by \eqref{E:tria2}, we have  $h_3\in Z(\g)$.
\end{proof}

\begin{lemma}\label{L:HRh}
Suppose that $\g$ is nilpotent and that $\dim \g=6$ and $\dim H=3$. Then the following conditions are equivalent:
\begin{enumerate}
\item  $\bH\cong \R\oplus \h_3$,
\item $\dim Z(\g)=2$,
\item $Z(\g)=\Span(h_3,v_3)$.
\end{enumerate}
\end{lemma}

\begin{proof} (c) $\implies$ (b) is obvious and Lemma~\ref{L:ceng} gives (b) $\implies$ (c).

(a) $\implies$ (c).
The derived algebra of $\bH$ has dimension one, so $[\bH,\bH]=\Span([h_1,h_2])$. In particular,
$H\cap[\bH,\bH]=0$. Hence $[v_3,h_i]=0$ for all $i=1,2,3$, and $[h_3,h_i]=0$ for all $i=1,2$. It follows from \eqref{E:vs} that $Z(\g)=\Span(h_3,v_3)$.

(b) $\implies$ (a). We prove the contrapositive. Suppose $\bH\cong \f_4$. So $\dim Z(\bH)=1$.  Note that by Lemma~\ref{L:ceng},
$Z(\g)\subseteq \Span(h_3,v_3)\subseteq \bH$. Hence $ Z(\g)\subseteq Z(\bH)$ and so  $\dim Z(\g)=1$.
\end{proof}

\begin{lemma}\label{L:centers}
Suppose that $\g$ is nilpotent and that $\dim \g=6$ and $\dim H=3$. Then $Z(\bH)= Z(\g)$.
\end{lemma}

\begin{proof} From Lemma~\ref{L:ceng}, $\dim Z(\g)\le 2$.
If $\dim Z(\g)=1$, then $\bH\cong \f_4$ by Lemma~\ref{L:HRh}. Thus $\bH$ has a one dimensional center, which is a characteristic ideal. So, as $\bH$ is an ideal of $\g$, for all $z\in \g$ we have $[z,Z(\bH)]\subseteq Z(\bH)$ and so, as $\g$ is nilpotent, $[z,Z(\bH)]=0$. Hence as $\dim Z(\g)=1$, we have $Z(\bH) = Z(\g)$. 

If $\dim Z(\g)=2$, then  by Lemma~\ref{L:HRh}, $\bH\cong \R\oplus \h_3$ and so $Z(\bH)$ has dimension two. By  Lemma~\ref{L:ceng}, $Z(\g)=\Span(h_3,v_3)\subseteq \bH$, so $Z(\g)\subseteq Z(\bH)$. Thus, for dimension reasons, $Z(\bH)= Z(\g)$.
\end{proof}

\begin{lemma}\label{L:der}
Suppose that $\g$ is nilpotent and that $\dim \g=6$ and $\dim H=3$. Then $\Span(h_2,h_3,v_3)$ is an ideal of $\g$ with Abelian quotient. In particular, the derived algebra $[\g,\g]$ is contained in $\Span(h_2,h_3,v_3)$ and hence has dimension at most three. 
\end{lemma}

\begin{proof} Let $\m=\Span(h_2,h_3,v_3)$. We will show that $[\g,\g]\subseteq\m$, which shows that $\m$ is an ideal and $\g/\m$ is Abelian, as required. Note that $[V,V]\subseteq\m$ and $[V,H]\subseteq\m$, by \eqref{E:tria2} and \eqref{E:vs}. So it suffices to show that $[H,H]\subseteq\m$.

If $\bH\cong \R\oplus \h_3$, then
 $\bH$ is two-step nilpotent and so $[\bH,\bH]\subseteq Z(\bH)$. So $[H,H]\subseteq[\bH,\bH]\subseteq\m$,  since $Z(\bH)=Z(\g)=\Span(h_3,v_3)$, by Lemmas~\ref{L:HRh} and \ref{L:centers}. So we may suppose that $\bH\cong \f_4$. From \eqref{E:tria2}, $\m$ is a Lie algebra in which the element $h_3$ is central with $[h_2,v_3]=\lambda h_3$ for some $\lambda\in\R$. If
 $\lambda=0$, then $\m$ is Abelian. But $\f_4$ has a unique 3-dimensional Abelian subalgebra, which contains the derived algebra $[\f_4,\f_4]$, and thus  $[\bH,\bH]\subseteq \m$.  
 So we may suppose $\lambda\not=0$. Thus $\m\cong \h_3$. By Lemma~\ref{L:ceng}, $Z(\g)\subseteq \Span(h_3,v_3)$. So as $[h_2,v_3]\not=0$ and $h_3$ is central in $\m$, we have $Z(\g)=\Span(h_3)$. Hence it remains to show that $[h_1,h_2]\in\m$.
 Let $[h_1,h_2]=v_3+a h_1+b h_2+c h_3$ for some $a,b,c\in \R$. 
As $\bH\cong \f_4$, which is three-step nilpotent, we have $[h_2,[h_1,h_2]]\in Z(\bH)=\Span(h_3)$. So 
$\pi_V[h_2, [h_1,h_2]]=0$. Hence, as $[h_2,v_3]=\lambda h_3\in H$ and $[h_2,h_3]=0$,
 \[
0=\pi_V[h_2, [h_1,h_2]]=\pi_V[h_2,v_3+a h_1+b h_2+c h_3]=-av_3,
\]
 so $a=0$. Thus $[h_1,h_2]\in\m$, as required.
 \end{proof}

\begin{remark} If the derived algebra $[\g,\g]$ of a nilpotent Lie algebra of dimension 6 has dimension $\le 3$, then $[\g,\g]$ is Abelian. This can be observed from the classification of 6-dimensional nilpotent Lie algebras \cite{deG}, or can be easily proved directly.
%If $[\g,\g]$ is a Heisenberg, its center is in the center of $\g$. So $\g/ Z([\g,\g])$ has a derived algebra of dimension two; so it is $L_{5,8}$; $[x_1,x_2]=x_4,\quad [x_1,x_3]=x_5$. Then consider Jacobi on $[[x_1,x_2],x_5$ in $\g$.
\end{remark}

\begin{lemma}\label{L:trick}
Suppose that $\g$ is nilpotent and that $\dim \g=6$ and $\dim H=3$. If $\dim Z(\g)=1$, then $[\g,\Span(h_2,h_3,v_3)]\not\subseteq Z(\g)$.
\end{lemma}

\begin{proof} If $\dim Z(\g)=1$, then $\bH\cong \f_4$ by Lemma~\ref{L:HRh}. As in the proof of Lemma~\ref{L:f4con}, $\Span(h_2,h_3,v_3)$ is the unique 3-dimensional Abelian ideal of $\bH$. Hence $[\bH,\Span(h_2,h_3,v_3)] =[\bH,\bH]$, which has dimension two  as $\bH\cong \f_4$. Thus $[\bH,\Span(h_2,h_3,v_3)]\not\subseteq Z(\g)$. The required result follows.
\end{proof}

\begin{proof}[Proof of Theorem \ref{T:main}]
Suppose that a 6-dimensional nilpotent Lie algebra $\g$ possesses an open respectful decomposition $(H,V)$ with $\dim H=3$. 
Parts (a), (b) are given by Lemmas~\ref{L:f4con} and \ref{L:ceng} respectively. Lemma~\ref{L:der}
gives $[\g,\g]\subseteq \Span(h_2,h_3,v_3)$ and Lemma~\ref{L:ceng} gives 
  $Z(\g)\subseteq \Span(h_3,v_3)$, so $[\g,\g]+Z(\g)\subseteq \Span(h_2,h_3,v_3)$, which has dimension $3$. This proves (c).
By Lemma~\ref{L:der}, $[\g,\g]\subseteq\Span(h_2,h_3,v_3)$, so if  $\dim [\g,\g]=3$, then $[\g,\g]=\Span(h_2,h_3,v_3)$.
If $\dim Z(\g)=1$, then Lemma~\ref{L:trick} gives  $[\g,\Span(h_2,h_3,v_3)]\not\subseteq Z(\g)$. This establishes (d).

 From the classification of real nilpotent Lie algebras of dimension 6, we see that up to isomorphism, there are exactly 14 nilpotent Lie algebras of dimension 6 satisfying the above four conditions; see Table~\ref{TableV3}, which uses the classification and notation of de Graaf \cite{deG}.  For each of these  14  algebras,  Table~\ref{TableV3} gives an open respectful decomposition $(H,V)$ with $\dim H=3$. (Note that there may be more than one such decomposition for each algebra; see Example~\ref{Ex:h3x2}). To see that the other algebras in de Graaf's classification do not possess an open respectful decomposition, note that:
\begin{itemize}
\item Algebras $L_{5,8}\oplus\R$ and $L_{6,25}$ fail condition (a).

\item Algebra $L_{6,26}$ and algebras of the form $\h\oplus \R^2$, where $\dim\h=4$, fail condition~(b).
(Note that $\h_3\oplus\h_3$ is isomorphic to the algebra $L_{6,22}(1)$).
% use $x_1-x_4,x_1+x_4,x_2+x_3,x_5+x_6,x_5-x_6$ ; $[x_1-x_4,x_2+x_3]=x_5+x_6$ $[x_1+x_4,x_2-x_3]=x_5-x_6$

\item Algebras $L_{5,i}\oplus\R$ for $i=6,7,9$, and 
$L_{6,i}$ for $i=14,15,16,17,18$, and $L_{6,21}(\epsilon)$ for $\epsilon\in\{-1,0,1\}$, fail condition (c).

\item Algebras $L_{6,19}(\pm1)$ and $L_{6,20}$ fail condition (d).
\end{itemize}
\end{proof}

\begin{table}
\begin{tabular}{ c | c | c | c | c | c }
  \hline
Algebra & Relations & $H$ basis & $V$ basis & $\bH$& $\V$    \\\hline
$L_{5,4}\oplus\R$ & $[x_1,x_2]=x_5,[x_3,x_4]=x_5$ & $x_1,x_2,x_5+x_6$ & $x_3,x_4,x_5-x_6$ & $\R\oplus \h_3$ & $\h_3$  \\\hline
$L_{5,5}\oplus\R$ & $[x_1,x_2]=x_3,[x_1,x_3]=x_5,[x_2,x_4]=x_5$ & $x_1,x_3,x_5+x_6$ & $x_2,x_4,x_6$ & $\R\oplus \h_3$ & $\h_3$  \\\hline
$L_{6,10}$ & $[x_1,x_2]=x_3,[x_1,x_3]=x_6,[x_4,x_5]=x_6$ & $x_1,x_2,x_6$ & $x_3,x_4,x_5$ & $\f_4$ & $\R^3$  \\\hline
$L_{6,11}$ & $[x_1,x_2]=x_3,[x_1,x_3]=x_4,[x_1,x_4]=x_6,$ & $x_1,x_3,x_6$ & $x_2,x_4,x_5$ & $\f_4$ & $\R^3$  \\
 & $[x_2,x_3]=x_6,[x_2,x_5]=x_6$ &  &  & &   \\\hline
$L_{6,12}$ & $[x_1,x_2]=x_3,[x_1,x_3]=x_4,$ & $x_1,x_3,x_6$ & $x_2,x_4,x_5$ & $\f_4$ & $\R^3$  \\
 & $[x_1,x_4]=x_6,[x_2,x_5]=x_6$ &  &  &  &   \\\hline
$L_{6,13}$ & $[x_1,x_2]=x_3,[x_1,x_3]=x_5,[x_2,x_4]=x_5,$ & $x_1,x_3,x_6$ & $x_2,x_4,x_5+x_6$ & $\f_4$ & $\h_3$  \\
 & $[x_1,x_5]=x_6,[x_3,x_4]=x_6$ &  &  &  &   \\\hline
$L_{6,19}(0)$ & $[x_1,x_2]=x_4,[x_1,x_3]=x_5,[x_2,x_4]=x_6$ & $x_2,x_4,x_5$ & $x_1,x_3,x_6$ & $\R\oplus\h_3$ & $\R^3$  \\\hline
$L_{6,22}(\epsilon)$ & $[x_1,x_2]=x_5,  [x_1,x_3]=x_6,$ & $x_1,x_2,x_6$ & $x_3,x_4,x_5+x_6$ & $\R\oplus\h_3$ & $\h_3$  \\
$\epsilon=0, \pm 1$ & $ [x_2,x_4]=\epsilon x_6,  [x_3,x_4]=x_5$ & &  &  &   \\\hline
$L_{6,23}$ & $[x_1,x_2]=x_3,[x_1,x_3]=x_5,$ & $x_1,x_3,x_6$ & $x_2,x_4,x_5+x_6$ & $\R\oplus\h_3$ & $\h_3$  \\ 
& $[x_1,x_4]=x_6,[x_2,x_4]=x_5$ &  &  &  &   \\ \hline
$L_{6,24}(\epsilon)$  & $[x_1,x_2]=x_3,[x_1,x_3]=x_5,[x_1,x_4]=\epsilon x_6,$ & $x_1,x_3,x_6$ & $x_2,x_4,x_5+x_6$ & $\R\oplus\h_3$ & $\h_3$  \\  
$\epsilon=0, \pm 1$  & $[x_2,x_3]=x_6,[x_2,x_4]=x_5$ &  &  &  &   \\  \hline
\end{tabular}
\bigskip
\caption{The 14 algebras of Theorem \ref{T:main}}\label{TableV3}
 \end{table}

\begin{proof}[Proof of Corollary \ref{cor}] Applying Lemma~\ref{L:JI}\eqref{L:H} to $H$ and $V$,
we have $\dim H=\dim V=3$.
As above, choose a basis $\{h_1,h_2,h_3\}$ for $H$ and $\{v_1,v_2,v_3\}$ for $V$ such that $\pi_H[h_1,h_2]=v_3$ and $\pi_H[v_1,v_2]=h_3$.
Since $H$ respects $V$, Lemma~\ref{L:VHV6} gives $\bH\cong \R\oplus \h_3$. 
Hence Lemma~\ref{L:HRh} gives  $Z(\g)=\Span(h_3,v_3)$. As $(H,V)$ is mutually respectful, $[h,v]=0$ for all $h\in H, v\in V$. So the only nontrivial relations are
\[
[h_1,h_2]=ah_3+v_3,\qquad [v_1,v_2]=h_3+bv_3,
\]
for some $a,b\in \R$. If $ab=1$, then obviously $\g\cong \R\oplus \h_5$. If $ab\not=1$, then the elements $ah_3+v_3,h_3+bv_3$ are linearly independent and it follows that $\g\cong \h_3\oplus \h_3$.
\end{proof}

%%%%%%%%%%%%%%%%%%%%%%%%%%%%%%%%%%%%%%%%%%%%%%%%%%%%
\bibliographystyle{amsplain}

\begin{thebibliography}{}


\bibitem{C}
Grant Cairns, \emph{A general description of totally 
geodesic foliations},
T\^ohoku Math. J.
\textbf{38} (1986),  37--55.

\bibitem{CHGN}
Grant Cairns, Ana Hini\'c Gali\'c, and Yuri Nikolayevsky, \emph{Totally
  geodesic subalgebras of nilpotent Lie algebras}, J. Lie Theory \textbf{23} (2013), no.~4, 1023--1049.
  
\bibitem{CHGN2}
\bysame, \emph{Totally
  geodesic subalgebras of filiform nilpotent Lie algebras}, J. Lie Theory \textbf{23} (2013), no.~4, 1051--1074.

\bibitem{CHNT}
Grant Cairns, Ana Hini\'c Gali\'c, Yuri Nikolayevsky and Ioannis Tsartsaflis,
\emph{Geodesic bases for Lie algebras},
Linear Multilinear Algebra \textbf{63} (2015),  no.~6, 1176--1194.

\bibitem{CLNN}
Grant Cairns, Nguyen Thanh Tung Le, Anthony Nielsen and Yuri Nikolayevsky, \emph{On the existence of orthonormal geodesic bases for Lie algebras},  Note di Matematica  \textbf{33} (2013), no.~2, 11--18.



\bibitem{CM}{G. Cairns and P. Molino}, \emph{Weakly involutive totally geodesic distributions of constant rank},   
in {Essays on geometry and related topics, {V}ol. 1},
{Monogr. Enseign. Math.}
  Vol. 38,
  177--204,
 Enseignement Math., Geneva,
  2001.


\bibitem{deG}{Willem A. de Graaf},
\emph{Classification of 6-dimensional nilpotent Lie algebras over fields of characteristic not 2},
  {Linear Algebra Appl.}
 \textbf{309} (2007),
    no.~2, 640--653.

 \bibitem{Du}
Zden{\v{e}}k Du{\v{s}}ek, \emph{The existence of homogeneous geodesics in
  homogeneous pseudo-{R}iemannian and affine manifolds}, J. Geom. Phys.
  \textbf{60} (2010), no.~5, 687--689.

\bibitem{Ebe}
Patrick Eberlein,
\emph{Geometry of 2-step nilpotent groups with a left invariant metric. II}, Trans. Amer. Math. Soc., \textbf{343}(1994), no.~2, 805--828.

\bibitem{F} Edmond Fedida, 
\emph{Sur les feuilletages de Lie},
C. R. Acad. Sci. Paris S\'er. A-B {\bf 272} (1971), A999--A1001. 

\bibitem{Ka}
V.~V. Ka{\u\i}zer, \emph{Conjugate points of left-invariant metrics on {L}ie
  groups}, Soviet Math. (Iz. VUZ) \textbf{34} (1990), no.~11, 32--44.


\bibitem{KS}
Old{\v{r}}ich Kowalski and J{\'a}nos Szenthe, \emph{On the existence of
  homogeneous geodesics in homogeneous {R}iemannian manifolds}, Geom. Dedicata
  \textbf{81} (2000), no.~1-3, 209--214, Erratum, Geom. Dedicata, {\bf 84}
  (2001), {no. 1-3}, {331--332}.

\bibitem{M} Pierre Molino,
\emph{Riemannian Foliations},
Birkh\"auser, 1988.


\bibitem{ML} Miguel-C. Mu\~noz-Lecanda, 
\emph{On some aspects of the geometry of non integrable distributions and applications}, 
J. Geom. Mech.  \textbf{10} (2018), no.~4, 445--465. 

\end{thebibliography}
{}

\end{document}